\documentclass[reqno]{amsart}

\usepackage[english]{babel}
\usepackage[T1]{fontenc}
\usepackage[utf8]{inputenx}

\usepackage{a4wide}
\usepackage{color}
\usepackage{mathrsfs}
\usepackage{mathtools}
\usepackage{tikz-cd}
\usepackage{amsmath}
\usepackage{amssymb}
\numberwithin{equation}{section}
\usepackage[colorlinks,citecolor=green,linkcolor=red]{hyperref}
\usepackage{hyphenat}
\usepackage{todonotes}
\usepackage{enumitem}

\newcommand{\N}{\mathbb{N}}
\newcommand{\R}{\mathbb{R}}
\renewcommand{\d}{{\mathrm d}}
\newcommand{\nchi}{{\raise.3ex\hbox{\(\chi\)}}}
\newcommand{\fr}{\penalty-20\null\hfill\(\blacksquare\)}
\newcommand{\mm}{\mathfrak m}
\newcommand{\Lip}{{\rm Lip}}
\newcommand{\lip}{{\rm lip}}
\newcommand{\LIP}{{\rm LIP}}

\newcommand{\aint}[2][]{
	\ifthenelse{\equal{#1}{}}
					{
\mathchoice%
      {\mathop{\kern 0.2em\vrule width 0.6em height 0.69678ex depth -0.58065ex
              \kern -0.8em \intop}\nolimits_{\kern -0.45em#2}^{#1}}%
      {\mathop{\kern 0.1em\vrule width 0.5em height 0.69678ex depth -0.60387ex
              \kern -0.6em \intop}\nolimits_{#2}^{#1}}%
      {\mathop{\kern 0.1em\vrule width 0.5em height 0.69678ex depth -0.60387ex
              \kern -0.6em \intop}\nolimits_{#2}^{#1}}%
      {\mathop{\kern 0.1em\vrule width 0.5em height 0.69678ex depth -0.60387ex
              \kern -0.6em \intop}\nolimits_{#2}^{#1}}}%
					{%
\mathchoice%
      {\mathop{\kern 0.2em\vrule width 0.6em height 0.69678ex depth -0.58065ex
              \kern -0.8em \intop}\nolimits_{\kern -0.45em#1}^{#2}}%
      {\mathop{\kern 0.1em\vrule width 0.5em height 0.69678ex depth -0.60387ex
              \kern -0.6em \intop}\nolimits_{#1}^{#2}}%
      {\mathop{\kern 0.1em\vrule width 0.5em height 0.69678ex depth -0.60387ex
              \kern -0.6em \intop}\nolimits_{#1}^{#2}}%
      {\mathop{\kern 0.1em\vrule width 0.5em height 0.69678ex depth -0.60387ex
              \kern -0.6em \intop}\nolimits_{#1}^{#2}}}}

\newtheorem{theorem}{Theorem}[section]
\newtheorem{corollary}[theorem]{Corollary}
\newtheorem{lemma}[theorem]{Lemma}
\newtheorem{proposition}[theorem]{Proposition}
\newtheorem{definition}[theorem]{Definition}
\newtheorem{remark}[theorem]{Remark}

\makeatletter
\@namedef{subjclassname@2020}{\textup{2020} Mathematics Subject Classification}
\makeatother

\title{Abstract and concrete tangent modules on Lipschitz differentiability spaces}
\author{Toni Ikonen}
\address{Department of Mathematics and Statistics,
P.O.\ Box 35 (MaD), FI-40014 University of Jyvaskyla}
\email{toni.m.h.ikonen@jyu.fi}

\author{Enrico Pasqualetto}
\address{Scuola Normale Superiore, Piazza dei Cavalieri 7,
56126 Pisa, Italy}
\email{enrico.pasqualetto@sns.it}

\author{Elefterios Soultanis}
\address{Radboud University, Department of Mathematics, PO Box 9010, Postvak 59, 6500 GL Nijmegen, The Netherlands}
\email{elefterios.soultanis@gmail.com}

\begin{document}
\date{\today} 
\begin{abstract}
We construct an isometric embedding from Gigli's abstract tangent module into the concrete tangent module of a space admitting a (weak) Lipschitz differentiable structure, and give two equivalent conditions which characterize when the embedding is an isomorphism. 
Together with arguments from \cite{BKO19}, this equivalence is used to show that the $\Lip-\lip$ -type condition $\lip f\le C|Df|$
self-improves to $\lip f =|Df|$.

We also provide a direct proof of a result in \cite{GP16} that, for a space with a strongly rectifiable decomposition, Gigli's tangent module admits an isometric embedding into the so-called Gromov--Hausdorff tangent module, without any a priori reflexivity assumptions.
\end{abstract}
\maketitle
\section{Introduction}

\subsection{Overview}
To study spaces with synthetic Ricci curvature bounds, Gigli \cite{Gigli14} developed a notion of abstract (normed) tangent module $L^2(TX)$ over a metric measure space $X=(X,d,\mm)$, that is, a complete separable metric space $(X,d)$ equipped with a Radon measure $\mm$ which is finite on bounded sets. The tangent (and cotangent) module allows for a first order differential calculus in this very general setting, at the cost of being highly abstract. Indeed, ``vector fields'' in $L^2(TX)$ are defined as elements in the module dual of an abstract object spanned by formal differentials of Sobolev functions (the cotangent module).

A level of concreteness can be recovered under some rectifiability conditions on $X=(X,d,\mm)$. In \cite{GP16} the second author and Gigli considered spaces admitting a \textbf{strongly rectifiable decomposition} (called $\mm$-rectifiable in \cite{GP16}). Assuming that $W^{1,2}(X)$ is reflexive, they prove that $L^2(TX)$ isometrically embeds into the module $L^2(T_{GH}X)$ of 2-integrable sections of the \textbf{Gromov--Hausdorff} tangent bundle $T_{GH}X\coloneqq\bigsqcup_{k\in\N}A_k\times\R^k$, cf. \cite[Theorem 5.1]{GP16}. Here the Borel sets $A_k$ partition $X$, each $\mm|_{A_k}$ is $k$-rectifiable, and each fiber is equipped with the standard Euclidean norm. It follows that $W^{1,2}(X)$ is a Hilbert space (i.e.\ $X$ is infinitesimally Hilbertian) but, as conjectured in \cite{GP16}, assuming that $W^{1,2}(X)$ is reflexive is unnecessary and in fact already follows from having a strongly rectifiable decomposition (see Corollary \ref{cor:reflexive}). We also mention \cite{pas-luc18}, where the authors show that a finitely generated normed module can be viewed as the space of sections of a suitable measurable Banach bundle, but the latter is in general unrelated to the structure of the underlying space.

More generally, the work of Cheeger \cite{Cheeger00} shows that the reflexivity of $W^{1,p}(X)$ is implied by the existence of chart maps on $X$ for which a generalized Rademacher theorem holds. This naturally leads to the notion of \textbf{Lipschitz differentiability spaces} (abbreviated LDS) which has been the focus of extensive research over the recent decade \cite{Keith04,Bate15,BateLi17,Schioppa14,Schioppa16:A,Schioppa16:B}. Lipschitz differentiability spaces include non-rectifiable spaces (e.g.\ the Heisenberg groups), but a close connection with rectifiability exists: \emph{a space is $n$-rectifiable if and only if it is a countable union of $n$-dimensional LDS's}, cf.\ \cite{BateLi17}. The construction in \cite{Cheeger00} (see also \cite{Keith04}) can be carried out if the underlying space decomposes into a countable union of  Lipschitz differentiability spaces, producing a ``concrete'' tangent bundle and differential. The subtle point here is that the countable union of LDS's is not generally an LDS, see the discussion in \cite[Introduction]{BateLi17} and in Section \ref{sec:lds}.
In this paper we call such spaces \textbf{weak Lipschitz differentiability spaces} (abbreviated weak LDS).

In this note we describe how the (adjoint of the) concrete differential on a weak LDS gives rise to an isometric embedding from the abstract into the concrete tangent module, see Theorem \ref{thm:isomA}.
This embedding is not always surjective (see the discussion after Theorem \ref{thm:module:LDS}); sufficient conditions include the Poincar\'e inequality \cite{Cheeger00} and its asymptotic forms, cf.\ Section \ref{sec:lds}. Theorem \ref{thm:isomB} gives two equivalent conditions,
in the spirit of Bate--Kangasniemi--Orponen's work \cite{BKO19}, which characterize surjectivity. The equivalence of the two conditions yields a self-improvement phenomenon similar to the self-improving of Keith's $\Lip - \lip$ condition \cite{Schioppa16:A}. Indeed, consider the condition
\begin{equation}\label{eq:BKO0}
\lip f\le C|Df|_p,\quad f\in \LIP_{bs}(X),
\end{equation}
adapted from \cite{BKO19}. Here $\LIP_{bs}(X)$ is the space of Lipschitz functions with bounded support on $X$, $\lip f$ is the pointwise Lipschitz constant and $|Df|_p$ the minimal $p$-weak upper gradient of $f$ (cf.\ Section \ref{sec:pre}). Note that condition \eqref{eq:BKO0} is reminiscent of, but stronger than, Keith's $\Lip-\lip$ condition \cite{Keith04} and (up to technical assumptions) both conditions imply the existence of a Lipschitz differentiable structure, see \cite[Theorem 1.4]{BKO19} and \cite[Corollary 10.5]{Bate15} respectively. We show that \eqref{eq:BKO0} characterizes the surjectivity of the isometric embedding and consequently self-improves to an equality, cf.\ Theorems \ref{thm:isomB} and \ref{thm:module:LDS}.

\subsection{Statement of results}
Let $X$ be a metric measure space. A pair $( U, \varphi )$, consisting of a Borel set $U \subset X$ with $\mm( U ) > 0$ and a Lipschitz function $\varphi \colon X \rightarrow \mathbb{R}^{k}$, is called a {\bf (strong) chart} of dimension $k$ if, for every $f \in \LIP(X)$ and $\mm$-a.e.\ $x \in U$, there exists a unique linear map $\d_{x}f \colon \mathbb{R}^{k} \rightarrow \mathbb{R}$ satisfying
\begin{equation}
\label{eq:differential:OG}
\lip( f - \d_{x}f \circ \varphi )(x)
=
0.
\end{equation}
A metric measure space is called a \textbf{Lipschitz differentiability space} (LDS) if it can be covered by a countable collection of charts (of arbitrary dimension). Our first result illustrates the self-improving property of \eqref{eq:BKO0}. In the statement, a \textbf{modulus of continuity} is a continuous increasing function $\omega\colon[0,\infty)\to [0,\infty)$ with $\omega(0)=0$.
\begin{theorem}\label{thm:module:LDS}
Let $X$ be a metric measure space and $\omega = \left\{ \omega_{x} \right\}_{ x \in X }$ a collection of moduli of continuity with the following property for some $1 < p < \infty$: For every $f \in \LIP_{bs}(X)$
\begin{equation}
	\label{eq:BKO}
	\lip(f)(x)
	\leq
	\omega_{x}( |Df|_{p}(x) ),
	\quad\text{ for }\mm\text{-a.e.\ }x\in X.
\end{equation}
Then $X$ is a Lipschitz differentiability space, and 
\begin{equation}\label{eq:lipeq}
\lip(f) = |Df|_{p'}\quad\mm\textrm{-a.e.\ in }X
\end{equation}
for every $f \in \LIP_{bs}(X)$ and every $p' \geq p$.
\end{theorem}
The conclusion that $X$ is a Lipschitz differentiability space is essentially contained in \cite[Theorem 1.4]{BKO19}; in this regard,
the only novelty is that the assumption of finite Hausdorff dimension is superfluous. Nevertheless, the main contribution of Theorem \ref{thm:module:LDS} is the self-improvement of \eqref{eq:BKO} to \eqref{eq:lipeq}. We point out that \eqref{eq:lipeq} does not necessarily extend to every $1 < p' <p$: for every $n \in \mathbb{N}$ and $\alpha > 0$, there exists a measure $\mu = \omega \mathcal{L}^{n}$ on $\mathbb{R}^{n}$ such that $X = ( \mathbb{R}^{n},|\cdot|, \mu )$ satisfies \eqref{eq:lipeq} for every $p > 1 + \alpha$, but $|Df|_{p'} \equiv 0$ for every $f \in \LIP_{bs}( X )$ and $p' \leq 1 + \alpha$; see \cite[Theorem 1.1]{DiMarinoSpeight15}.

The equality \eqref{eq:lipeq} was originally observed in PI spaces \cite[Sections 5 and 6]{Cheeger00}. Note that \eqref{eq:BKO} is not valid in every LDS: if $K \subset \mathbb{R}^{n}$ is a Cantor set with $\mathcal L^n(K)>0$, then $X = ( K, | \cdot |, \mathcal{L}^{n}|_{K} )$ is an LDS with $|Df|_p\equiv 0$ for any $f\in \LIP(X)$ and $p\ge 1$ (since $X$ does not contain any rectifiable paths). By comparison, Keith's $\Lip-\lip$ condition -- which satisfies an analogous self-improvement property stated in Theorem \ref{thm:module:LDS} -- characterizes LDS's (for pointwise doubling measures). This demonstrates that \eqref{eq:BKO} guarantees better connectivity properties (and also the pointwise doubling property of measures). Some sufficient conditions for \eqref{eq:BKO} are discussed in Section \ref{sec:lds}.

To obtain Theorem \ref{thm:module:LDS} we construct an isometric embedding from the abstract tangent module into the concrete one. The construction is carried out on {\bf weak Lipschitz differentiability spaces} (weak LDS), that is, metric measure spaces which can be covered (up to a null-set) by {\bf (weak) charts} $(U,\varphi)$ consisting of a Borel set \(U\) with $\mm(U)>0$ and a Lipschitz function $\varphi\colon X\to \R^k$ such that, for any $f \in \LIP(X)$ and $\mm$-a.e.\ $x \in U$, there exists a unique linear map $\d_xf:\R^k\to \R$ satisfying
\begin{align}\label{eq:lds}
\lip(f|_U-\d_xf\circ\varphi|_{U})(x)=0.
\end{align}
A collection $\mathscr A=\{(U_n,\varphi_n)\}$ of $k_n$-dimensional (weak) charts partitioning $X$ up to a null-set is called an {\bf atlas} of $X$, and gives rise to a concrete tangent bundle $T_{\mathscr A}X=\bigsqcup_{n\in\N} U_n\times\R^{k_n}$ where, for $\mm$-a.e.\ $x\in U_n$, $\{x\}\times \R^{k_{n}}$ is equipped with the dual norm of $L \mapsto \lip(L\circ\varphi_n|_{U_n})$, denoted $|\cdot|_{\mathscr A,x}$. The associated space of $q$-integrable sections ($1<q<\infty$) is independent of the chosen atlas, up to isometric isomorphism, and we denote it $\Gamma_q(TX)$; see Section \ref{sec:conc}.

We refer the reader to \cite{Gigli14} for the construction of the Gigli tangent module $L^{q}( TX )$, and mention here that elements $v \in L^{q}( TX )$ are in one-to-one correspondence with linear maps $V \colon N^{1,p}( X ) \rightarrow L^{1}( \mm )$ ($1/p + 1/q = 1$) for which there exists $\ell \in L^{q}( \mm )$ with $| V( f ) | \leq \ell | Df |_{p}$ for every $f \in N^{1,p}(X)$. Here $N^{1,p}(X)$ refers to the Sobolev space on $X$ with exponent $1 < p < \infty$, see Section \ref{sec:newton}.

\begin{theorem}\label{thm:isomA}
	Let $X=(X,d,\mm)$ be a weak Lipschitz differentiability space, and $1<p,q<\infty$, $1/p+1/q=1$. Then there exists an isometric embedding $\iota\colon L^q(TX)\to \Gamma_q(TX)$ of normed modules
	satisfying
	\begin{align*}
	\d_xf(\iota(V)(x))=V(f)(x)\quad \mm\textrm{-a.e.\ }x\in X
	\end{align*}
	for every $f\in \LIP_{bs}(X)$ and $V\in L^q(TX)$.
\end{theorem}	
	
\begin{theorem}\label{thm:isomB}
The embedding $\iota$ in Theorem \ref{thm:isomA} is an isometric isomorphism if and only if one of the following equivalent conditions hold.
	\begin{itemize}
		\item[(1)] We have $\lip(f|_U) = |Df|_p$ $\mm$-a.e.\ in $U$ for each weak chart $(U,\varphi)$ and every $f\in \LIP_{bs}(X)$;
		\item[(2)] There is a collection $\omega=\{\omega_x\}_{x\in X}$ of moduli of continuity so that $\lip(f|_U)\le \omega(|Df|_p)$ $\mm$-a.e.\ on $U$ for each weak chart $(U,\varphi)$ and every $f\in \LIP_{bs}(X)$.
	\end{itemize}
\end{theorem}
As the proof will show, conditions (1) and (2) are satisfied if and only if they are satisfied for each chart in a given atlas. We point out the following corollary which, although implicitly well known, seems to not have been explicitly stated in the literature.
\begin{corollary}\label{cor:reflexive}
	If $X=(X,d,\mm)$ is a weak LDS metric measure space, then $N^{1,p}(X)$ is reflexive for any $1<p<\infty$.
\end{corollary}

Theorem \ref{thm:module:LDS} is obtained from the equivalence of (1) and (2) in Theorem \ref{thm:isomB} and the fact that the pointwise Lipschitz constant satisfies the same locality properties as the minimal weak upper gradient in an LDS: 
\begin{equation}\label{eq:liploc}
	\lip( f ) = 0\quad\mm\textrm{-a.e.\ in }U,
\end{equation}
for any Borel $U \subset X$ and $f \in \LIP(X)$ vanishing on $U$ (compare \eqref{eq:sobloc}). Indeed, a weak LDS whose measure vanishes on porous sets is an LDS and this null-porosity yields \eqref{eq:liploc}. We refer to Section \ref{sec:lds} for details, and formulate the following variant of Theorem  \ref{thm:module:LDS} in the weak LDS setting.
 
\begin{theorem}\label{cor:BKO}
	Let $X=(X,d,\mm)$ be a metric measure space and $1<p<\infty$. Suppose there is a collection $\omega=\{\omega_x\}$ of moduli of continuity, and Borel sets $\{U_k\}$ partitioning $X$ up to a null-set, such that $\lip(f|_{U_k})\le \omega(|Df|_p)$  $\mm$-a.e.\ on $U_k$ for every $f\in \LIP_{bs}(X)$. Then $X$ is a weak LDS satisfying (1) in Theorem \ref{thm:isomB} for every $p' \geq p$.
\end{theorem}

We finish the introduction by briefly describing the starting point of this work: a proof of \cite[Theorem 5.1]{GP16} without the a priori reflexivity of the Sobolev space. A (weak) $k$-dimensional chart $(U,\varphi)$ of a metric measure space $X$ is {\bf rectifiable} if $\mm|_{U}\ll \mathcal H^{k}|_{U}$ and $\varphi\colon U\to \R^{k}$ is a bi-Lipschitz embedding. We say that $X$ admits a \textbf{strongly rectifiable decomposition} if it can be covered up to a null-set by $(1+\varepsilon)$-bi-Lipschitz rectifiable charts (of arbitrary dimension) for every $\varepsilon>0$.

\begin{theorem}\label{thm:GH}
	If a metric measure space $X=(X,d,\mm)$ admits a strongly rectifiable decomposition, then $\Gamma_2(TX)$ is isometrically isomorphic to $L^2(T_{GH}X)$ as normed modules.
\end{theorem}
As an immediate corollary of Theorems \ref{thm:isomA} and \ref{thm:GH} we obtain \cite[Theorem 5.1]{GP16}.
\begin{corollary}\label{cor:GH}
If $X=(X,d,\mm)$ admits a strongly rectifiable decomposition, then $L^2(TX)$ embeds isometrically in $L^2(T_{GH}X)$ as a normed module.
\end{corollary}
\section{Preliminaries}\label{sec:pre}
Throughout the paper, a \textbf{metric measure space} is a triple \((X,d,\mm)\), where \((X,d)\) is a complete separable metric space and \(\mm\geq 0\) is a Radon measure which is finite on bounded sets. We denote by $\Lip(f)$ the Lipschitz constant of a function $f\colon X\to\R$ and by $\LIP(X)$ (resp.\ $\LIP_{bs}(X)$) the space of Lipschitz functions (resp.\ Lipschitz functions with bounded support). The \textbf{pointwise Lipschitz constant} of $f\in \LIP(X)$ is given by
\[
\lip f(x)\coloneqq\varlimsup_{y\to x}\frac{\big|f(y)-f(x)\big|}
{d(y,x)},\quad\text{ for every accumulation point }x\in X,
\]
and \(\lip f(x)\coloneqq 0\) elsewhere.

\subsection{Sobolev spaces on metric measure spaces}\label{sec:newton}
Let $X=(X,d,\mm)$ be a metric measure space, and \(p\in[1,\infty]\). A family $\Gamma \subset \mathcal{C}( \left[0, 1\right], X )$ of paths is \textbf{\(p\)-negligible} if there exists a Borel function $g \in L^{p}( \mm )$ such that
\begin{equation*}
	\infty
	=
	\int_{ 0 }^{ 1 }
	g( \gamma_{t} )
	|\dot\gamma_{t}|
	\,\d t
\end{equation*}
for every absolutely continuous $\gamma \in \Gamma$.

\begin{definition}[Sobolev space]\label{def:sobolev}
A Borel function $g \colon X \rightarrow \left[0, \infty\right]$ is an \textbf{upper gradient} of a function $f \colon X \rightarrow \mathbb{R}$, if
\begin{equation}
    \label{eq:upper:gradient:ineq}
    \big| f( \gamma_{1} ) - f( \gamma_{0} ) \big|
    \leq
    \int_{ 0 }^{ 1 }
        g( \gamma_{t} )
        |\dot\gamma_{t}|
    \,\d t
\end{equation}
for every absolutely continuous $\gamma \colon \left[0, 1\right] \to X$. We say that $f \in N^{1,p}(X)$ if $f \in L^{p}(\mm)$ and there exists an upper gradient $g \in L^{p}(\mm)$ of $f$.
\end{definition}

If \eqref{eq:upper:gradient:ineq} holds for all paths outside a $p$-negligible path family, $g$ is said to be a \textbf{$p$-weak upper gradient} of $f$. Every Newton--Sobolev function $f\in N^{1,p}(X)$ admits a minimal $p$-weak upper gradient (see \cite{HKST15}), which we denote by $|Df|_p$ to emphasize its dependence on the exponent $p$. We refer the reader to \cite{HKST15} for further details.

An important feature of minimal \(p\)-weak upper gradients
is their \textbf{locality property}
(cf.\ \cite[Proposition 6.3.22]{HKST15}): if
\(f\in N^{1,p}(X)\) and \(E\subseteq X\) is a Borel set such that
\(f=0\) \(\mm\)-a.e.\ on \(E\), then 
\begin{equation}\label{eq:sobloc}
|Df|_p=0\quad\mm\text{-a.e.\ on }E.
\end{equation}

Here Newton--Sobolev functions are everywhere defined and $\|f\|_{1,p}\coloneqq(\|f\|_p^p+\||Df|_p\|_p^p)^{1/p}$ is only a seminorm on $N^{1,p}(X)$. The quotient, called the {\bf Newton--Sobolev space}, is a Banach space, but the equivalence classes of functions are determined by equality outside a set of zero $p$-capacity rather than a.e.\ equality, cf.\ \cite[Chapter 7]{HKST15}. The Newton--Sobolev space over a metric measure space agrees with several other notions, in particular the one obtained by relaxation of the $p$-integral of the pointwise Lipschitz constant.
\begin{theorem}[Density in energy of Lipschitz functions
\cite{AmbrosioGigliSavare11-3}]\label{thm:density_energy}
Let $1<p<\infty$. For every \(f\in N^{1,p}(X)\)
there exists a sequence \((f_n)_n\subseteq\LIP_{bs}(X)\) such
that \(f_n\to f\) and \(\lip(f_n)\to|Df|_p\) in \(L^p(\mm)\).
\end{theorem}
\subsection{Abstract tangent and cotangent module}
We assume the reader to be familiar with the language of \textbf{normed modules} on a given metric measure space \((X,d,\mm)\); cf.\ \cite{Gigli14,Gigli17} for a detailed account on this theory.
\begin{theorem}[Cotangent module]
	Let \((X,d,\mm)\) be a metric measure space, and \(1<p<\infty\). Then there exist a \(L^p(\mm)\)-normed
	\(L^\infty(\mm)\)-module \(L^p(T^*X)\) (called the 
	\textbf{abstract cotangent module}) and a linear map
	\(\d\colon N^{1,p}(X)\to L^p(T^*X)\) (called the \textbf{differential})
	such that:
	\begin{itemize}
		\item[\(\rm i)\)] \(|\d f|=|Df|_p\) holds \(\mm\)-a.e.\ on
		\(X\) for every \(f\in N^{1,p}(X)\).
		\item[\(\rm ii)\)] The family of all elements of the form
		\(\sum_{i=1}^n\nchi_{A_i}\d f_i\), where \((A_i)_i\)
		is a Borel partition of \(X\) and \((f_i)_i\subseteq N^{1,p}(X)\),
		is dense in \(L^p(T^*X)\).
	\end{itemize}
	Moreover, the couple \(\big(L^p(T^*X),\d\big)\) is uniquely determined, up to unique isometric isomorphism.
\end{theorem}
This result was originally proved in \cite[Section 2.2.1]{Gigli14} for \(p=2\). An easy adaptation to the case of an arbitrary \(p\in(1,\infty)\) (and to more general classes of Sobolev spaces) can be found in \cite[Theorem 3.2]{GP19}.

\medskip
The \textbf{abstract tangent module} \(L^q(TX)\) is defined as the
module dual of \(L^p(T^*X)\). Observe that \(q\) is the conjugate
exponent of \(p\) and that \(L^q(TX)\) is a \(L^q(\mm)\)-normed
\(L^\infty(\mm)\)-module.

\subsection{Submetries and adjoints between normed modules}
Let \(\mathscr M,\mathscr N\) be \(L^p(\mm)\)-normed \(L^\infty(\mm)\)-modules,
for some exponent \(p\in(1,\infty)\). We refer to continuous $L^{\infty}(\mm)$-linear maps $\Phi \colon \mathscr M \to \mathscr N$ as \textbf{morphisms}. A morphism
\(\Phi\colon\mathscr M\to\mathscr N\) is said to be a
\textbf{submetry} provided for every element \(w\in\mathscr N\) there exists
\(v\in\mathscr M\) such that \(\Phi(v)=w\) and
\[
|v|=|w|=\underset{v'\in\Phi^{-1}(w)}{\rm ess\,inf}|v'|,
\quad\mm\text{-a.e.\ on }X.
\]
In particular, \(\Phi\) is surjective and satisfies
\(\big|\Phi(v)\big|\leq|v|\) in the \(\mm\)-a.e.\ sense for every
\(v\in\mathscr M\).
Note that an injective submetry is automatically an isometric isomorphism.
\begin{proposition}[Adjoint operator]\label{prop:adjoint_operator}
	Let \((X,d,\mm)\) be a metric measure space. Let \(\mathscr M\),
	\(\mathscr N\) be \(L^p(\mm)\)-normed \(L^\infty(\mm)\)-modules, for
	some \(p\in(1,\infty)\). Let \(\Phi\colon\mathscr M\to\mathscr N\)
	be a given morphism. Then there exists a unique morphism
	\(\Psi\colon\mathscr N^*\to\mathscr M^*\), called the \textbf{adjoint operator}
	of \(\Phi\), such that
	\begin{equation}\label{eq:def_adjoint}
	\big\langle\Psi(\eta),v\big\rangle=\big\langle\eta,\Phi(v)\big\rangle,
	\quad\text{ for every }v\in\mathscr M\text{ and }\eta\in\mathscr N^*.
	\end{equation}
	Moreover, if \(\Phi\) is a submetry, then \(\Psi\) is an isometric embedding.
\end{proposition}
\begin{proof}
	Define \(\Psi\colon\mathscr N^*\to\mathscr M^*\) as in
	\eqref{eq:def_adjoint}. Being \(\Phi\) a morphism, there exists a constant
	\(C>0\) such that \(\big|\Phi(v)\big|\leq C|v|\) holds \(\mm\)-a.e.\ for
	every \(v\in\mathscr M\). Hence, for any \(\eta\in\mathscr N^*\) and
	\(v\in\mathscr M\) it holds
	\[
	\big|\big\langle\Psi(\eta),v\big\rangle\big|=
	\big|\big\langle\eta,\Phi(v)\big\rangle\big|\leq
	C|\eta||v|,\quad\mm\text{-a.e.\ on }X,
	\]
	which shows that \(\Psi\) is a well-defined morphism. Observe that
	it is uniquely determined by \eqref{eq:def_adjoint}.
	Now suppose \(\Phi\) is a submetry. Given any \(\eta\in\mathscr N^*\),
	we know from \eqref{eq:def_adjoint} that
	\begin{equation}\label{eq:prop_adjoint_aux}
	\big|\Psi(\eta)\big|=\underset{\substack{v\in\mathscr M: \\ |v|\leq 1
			\;\mm\text{-a.e.}}}{\rm ess\,sup}\big\langle\Psi(\eta),v\big\rangle=
	\underset{\substack{v\in\mathscr M: \\ |v|\leq 1
			\;\mm\text{-a.e.}}}{\rm ess\,sup}\big\langle\eta,\Phi(v)\big\rangle\leq
	|\eta|\underset{\substack{v\in\mathscr M: \\ |v|\leq 1
			\;\mm\text{-a.e.}}}{\rm ess\,sup}|v|\leq|\eta|,
	\quad\mm\text{-a.e.\ on }X,
	\end{equation}
	thus \(\big|\Psi(\eta)\big|\leq|\eta|\) holds \(\mm\)-a.e.\ on \(X\).
	To prove the converse inequality, note that
	given \(w\in\mathscr N\) with \(|w|\leq 1\) in the \(\mm\)-a.e.\ sense,
	there exists \(v\in\Phi^{-1}(w)\) such that \(|v|=|w|\leq 1\) holds
	\(\mm\)-a.e., so
	\[
	|\eta|=\underset{\substack{w\in\mathscr N: \\ |w|\leq 1\;\mm\text{-a.e.}}}
	{\rm ess\,sup}\langle\eta,w\rangle\leq
	\underset{\substack{v\in\mathscr M: \\ |v|\leq 1\;\mm\text{-a.e.}}}
	{\rm ess\,sup}\big\langle\eta,\Phi(v)\big\rangle\leq
	\underset{\substack{v\in\mathscr M: \\ |v|\leq 1\;\mm\text{-a.e.}}}
	{\rm ess\,sup}\big\langle\Psi(\eta),v\big\rangle\leq\big|\Psi(\eta)\big|,
	\quad\mm\text{-a.e.\ on }X.
	\]
	All in all, we have proven that \(\big|\Psi(\eta)\big|=|\eta|\) holds
	\(\mm\)-a.e.\ on \(X\), whence \(\Psi\) is an isometry.
\end{proof}
\subsection{Measure theory}\label{subsec:porousity}
Let $X = (X, d, \mm)$ be a metric measure space. Let $S \subset U \subset X$ be sets. We say that $S$ is \textbf{porous} in $U$ if for every $x \in S$, there exists a sequence $( x_{n} )_{n = 1}^{\infty} \subset U$ converging to $x$ and $\eta > 0$ such that $d( x_{n}, S ) > \eta d( x_{n}, x )$ for every $n \in \mathbb{N}$. We say that $S \subset X$ is \textbf{porous} if $S$ is porous in $X$.
\begin{remark}\label{rmk:por}{\rm
	Note that if $S$ is porous and $f(x)\coloneqq d(x,S)$,
	then $\lip f(x)>0$ for every $x\in S$.\fr}
\end{remark}

We say that $\mm$ is \textbf{infinitesimally doubling} if $\limsup_{r\to 0}\frac{\mm(B(x,2r))}{\mm(B(x,r))}<\infty$ $\mm$-a.e.\ $x\in X$.
\begin{theorem}[{\cite[Theorem 3.6]{MMPZ03}}]\label{thm:porous:infdoub}
If porous sets have zero $\mm$-measure, then $\mm$ is infinitesimally doubling.
\end{theorem}
We recall that if $\mm$ is infinitesimally doubling, then $X$ is a Vitali space. In particular, the Lebesgue differentiation theorem holds \cite[Section 3.4]{HKST15}.
\begin{lemma}[{\cite[Lemma 8.3]{Bate15}}]\label{lemm:finitedimension}
If $\mm$ is infinitesimally doubling, then $X$ can be covered up to a null set by sets of finite Hausdorff dimension.    
\end{lemma}
Note that the sets $Y\in D_\eta$ in \cite[Lemma 8.3]{Bate15} have $\dim_HY\le \log_2\eta/5$.
\section{Concrete tangent module and existence of an isometric embedding}

\subsection{Concrete tangent and cotangent module}\label{sec:conc}

Throughout this section, we fix a weak LDS $X=(X,d,\mm)$. An {\bf atlas } on $X$ is a countable family of charts whose domains partition $X$ up to a null-set. Given an atlas $\mathscr A=\{ (U_k,\varphi_k) \}_{k\in\N}$ on $X$, consider the measurable tangent and cotangent bundles
\[
T_{\mathscr A}^*X=\bigsqcup_{k\in\N}U_k\times(\R^{n_k})^*,\quad T_{\mathscr A}X=\bigsqcup_{k\in\N}U_k\times \R^{n_k}
\]
where, for $\mm$-a.e.\ $x\in U_k$,
\[
|L|_{\mathscr A,x}^*\coloneqq\lip(L\circ \varphi_k)(x),\quad |\xi|_{\mathscr A,x}\coloneqq\max\{ L(\xi):\ |L|_{\mathscr A,x}^*\le 1 \}
\]
for every $L\in (\R^{n_k})^*$ and $\xi\in \R^{n_k}$. Equip $T_{\mathscr A}X$ with the $\sigma$-algebra to which $S\subset T_{\mathscr A}X$ belongs if and only if $S\cap (U_k\times \R^{n_k})$ is a Borel set for every $k$, and denote $\pi\colon T_{\mathscr A}X\to X$, $(x,v)\mapsto x$ the bundle projection. A measurable section of $T_{\mathscr A}X$ is a measurable map $v\colon X\to T_{\mathscr A}X$ with $\pi\circ v(x)=x$ for all $x\in \bigcup_kU_k$. We define the analogous notions on $T_{\mathscr A}^*X$ in the obvious manner.

Let $\Gamma(T_{\mathscr A}X)$ and $\Gamma(T_{\mathscr A}^*X)$ denote the spaces of measurable sections of $T_{\mathscr A}X$ and $T_{\mathscr A}^*X$ respectively, considered up to \(\mm\)-a.e.\ equality. For $1<p<\infty$, define 
\begin{align*}
\Gamma_p(T_{\mathscr A}X)=\{v\in \Gamma(T_{\mathscr A}X) :\ \||v|_{\mathscr A}\|_p<\infty\},\quad\Gamma_p(T_{\mathscr A}^*X)=\{\omega\in \Gamma(T_{\mathscr A}^*X) :\ \||\omega|_{\mathscr A}^*\|_p<\infty\}.
\end{align*}
Note that $\Gamma_p(T_{\mathscr A}X)$ and $\Gamma_p(T_{\mathscr A}^*X)$ are $L^p(\mm)$-normed $L^\infty(\mm)$-modules over $(X,d,\mm)$.
\begin{remark}\label{rmk:simpledense}{\rm
The spaces $\Gamma_p(T_{\mathscr A}X)$ and $\Gamma_p(T_{\mathscr A}^*X)$ are generated by ``simple'' sections. That is, the subspace of sections taking only finitely many values is dense in each space.
\fr}\end{remark}

Let $\mathscr A'$ be another atlas. If $(U,\varphi)\in\mathscr A$, $(V,\psi)\in\mathscr A'$ and $\mm(U\cap V)>0$, then $\varphi$ and $\psi$ restrict to charts on $U\cap V$, and thus have the same dimension $n$. Moreover we have that
\begin{align}\label{eq:restrchart}
\lip (g|_{U\cap V})=\lip (g|_U)=\lip (g|_V)\quad \mm\textrm{-a.e.\ on }U\cap V
\end{align}
for every $g\in \LIP(X)$, see \cite[Corollary 2.7]{BS13}. Differentiating the components of one chart with respect to the other chart we obtain, for $\mm$-a.e.\ $x\in U\cap V$, a linear isomorphism $D_x\colon\R^n\to \R^n$ satisfying
\begin{align}\label{eq:ldccomp}
|L|_{\mathscr A',x}^*=|L\circ D_x|_{\mathscr A,x}^*,\quad|\xi|_{\mathscr A,x}=|D_x(\xi)|_{\mathscr A',x},\quad L\in (\R^n)^*,\ \xi\in \R^n.
\end{align}
Thus we have a measurable map $D\colon T_{\mathscr A}X\to T_{\mathscr A'}X$ which is an isometry on a.e.\ fiber, and induces isometric isomorphisms $D\colon\Gamma_p(T_{\mathscr A}X)\to \Gamma_p(T_{\mathscr A'}X)$, $D^*\colon\Gamma_p(T_{\mathscr A'}^*X)\to \Gamma_p(T_{\mathscr A}^*X)$. This gives rise to the  {\bf concrete tangent and cotangent modules}, denoted by $\Gamma_p(TX)$ and $\Gamma_p(T^*X)$, which are unique up to isometric isomorphism.
\begin{definition}\label{def:concretediff}
	For each $f\in \LIP(X)$, the measurable section $$\underline\d f\colon X\to T_{\mathscr A}^*X,\quad x\mapsto(x,\d_xf)$$ given by \eqref{eq:lds}, is called the {\bf concrete differential} of $f$, and the map $\underline\d\colon\LIP_{bs}(X)\to \Gamma_p(T^*X)$ is called the concrete differential.
\end{definition}
Note that for two atlases $\mathscr A,\mathscr A'$ the concrete differentials $\underline\d f,\underline\d'f$ satisfy $\underline\d f\circ D=\underline\d'f$, where $D\colon T_{\mathscr A}X\to T_{\mathscr A'}X$ is the isometric isomorphism constructed above. In particular the concrete differential is well-defined. Moreover it does not depend on $p$.

\begin{remark}\label{rmk:concdual}{\rm
It is not difficult to see that the module duals of the concrete tangent and
cotangent modules satisfy $$\Gamma_p(T^*X)^*=\Gamma_q(TX),\quad \Gamma_q(TX)^*=\Gamma_p(T^*X)$$ if $1/p+1/q=1$. In particular, both spaces are reflexive.
\fr}\end{remark}

\subsection{Submetry of cotangent modules}\label{sec:submetry}

In this section we construct an isometric embedding from the abstract tangent module into the concrete tangent module, and establish Theorem \ref{thm:isomA}. The embedding arises as the adjoint map of the $L^\infty(\mm)$-linear extension of the concrete differential. The next theorem states that this extension defines a submetry.

\begin{theorem}\label{thm:submetry}
Let $X=(X,d,\mm)$ be a weak LDS, and $1<p<\infty$. Then there exists a submetry $P\colon\Gamma_p(T^*X)\to L^p(T^*X)$ of normed modules, such that
\begin{equation}\label{eq:diffext}
P(\underline\d f)=\d f\quad \mm\text{-a.e.}
\end{equation}
for all $f\in \LIP_{bs}(X)$. Consequently the adjoint map $\iota\coloneqq P^*\colon L^q(TX)\to \Gamma_q(TX)$, where $1/p+1/q=1$, is an isometric embedding of normed modules.
\end{theorem}

Throughout this subsection, we fix an atlas $\mathscr A=\{ (U_k,\varphi_k)\}$ of $X$ and denote by $n_k$ the dimension of the chart $(U_k,\varphi_k)\in\mathscr A$. Note that, for any $f\in \LIP_{bs}(X)$, we have the equality
\begin{align}\label{eq:ldsnorm}
|\underline\d f|_{\mathscr A}^*=\sum_k\nchi_{U_k}\lip(f|_{U_k}) \quad \mm\text{-a.e.}
\end{align}
by \eqref{eq:lds}. For stating the next lemma, we say that an absolutely continuous path $\gamma \colon \left[0, 1\right] \to X$ has \textbf{positive length} in a Borel set $B \subset X$ if
\begin{equation*}
\int_{ 0 }^{ 1 }
\nchi_{ B }( \gamma_{t} )
|\dot\gamma_{t}|
\,\d t
>
0.
\end{equation*}
The collection $\Gamma_B^+$ of paths with positive length in $B$ is $p$-negligible for all $p$, whenever $\mm(B)=0$, cf.\ \cite[Lemma 5.2.15]{HKST15}.

\begin{lemma}\label{lem:concrete_diff_wug}
	Let \(N\coloneqq X\setminus\bigcup_k U_k\) and \(f\in\LIP(X)\). Then for every absolutely continuous $\gamma\notin \Gamma_N^+$ we have that
	\begin{equation*}
	\big|(f\circ\gamma)'_t\big|\leq\sum_{k}\nchi_{U_k}(\gamma_t)\,
	\lip(f|_{U_k})(\gamma_t)\,|\dot\gamma_t|,\quad\text{ for }
	\mathcal L^1\text{-a.e.\ }t\in[0,1].
	\end{equation*}
	In particular, for any exponent \(p\in(1,\infty)\) and any function
	\(f\in\LIP_{bs}(X)\subseteq N^{1,p}(X)\) it holds that
	\begin{equation}\label{eq:concrete_diff_wug21}
	|Df|_p\leq|\underline \d f|_{\mathscr A}^*\le \lip f \quad\mm\text{-a.e.\ on }X.
	\end{equation}
\end{lemma}
\begin{proof}
	Since $\Gamma_N^+$ is $p$-negligible, the second claim in the statement follows from the minimality of $|Df|_p$ and \eqref{eq:ldsnorm},
cf.\ \cite{Bjorn-Bjorn11}. To prove the first claim, we may assume that $|\dot\gamma_t|>0$ for a.e.\ $t$ since $|(f\circ\gamma)_t'|\le \lip f(\gamma_t)|\dot\gamma_t|$ for a.e.\ $t$.	
	Define \(I_k\coloneqq\big\{t\in[0,1]\,:\,\gamma_t\in U_k\big\}\) for every
	\(k\in\N\). Observe that \(\mathcal L^1\big([0,1]\setminus\bigcup_k I_k\big)=0\).
	Given any \(k\in\N\), fix a density point \(t\in I_k\) of \(I_k\) such that
	\(|\dot\gamma_t|\) and \((f\circ\gamma)'_t\) exist (almost every point of
	\(I_k\) has this property). Pick any sequence \((t_n)_n\subseteq I_k\)
	such that \(t_n\to t\). Then it holds that
	\[
	\big|(f\circ\gamma)'_t\big|=\lim_{n\to\infty}\frac{\big|f(\gamma_{t_n})-
		f(\gamma_t)\big|}{|t_n-t|}\leq\limsup_{n\to\infty}
	\frac{\big|f(\gamma_{t_n})-f(\gamma_t)\big|}{d(\gamma_{t_n},\gamma_t)}
	\lim_{n\to\infty}\frac{d(\gamma_{t_n},\gamma_t)}{|t_n-t|}
	\leq\lip(f|_{U_k})(\gamma_t)\,|\dot\gamma_t|.
	\]
	Thus the claim follows.
\end{proof}

\begin{proof}[Proof of Theorem \ref{thm:submetry}]
	The linear subspace 
	\[
	\mathcal V\coloneqq\bigg\{\sum_{j=1}^n\nchi_{A_j}\underline\d f_j\;
	\bigg|\;n\in\N,\,(A_j)_{j=1}^n\text{ bounded pairwise disjoint Borel sets in }X,\, (f_j)\subset\LIP_{bs}(X)\bigg\}
	\]
	is a dense linear subspace of \(\Gamma_p(T^*X)\). Indeed, by Remark \ref{rmk:simpledense}, we can choose the functions $f_j$ to be Lipschitz extensions with bounded support of the functions $L_j\circ\varphi_k\colon U_k\cap A_j\to\R$, where $L_j\in(\R^{n_k})^*$.
	We define \(P\colon\mathcal V\to L^p(T^*X)\) as
	\[
	P(\underline\omega)\coloneqq\sum_{j=1}^n\nchi_{A_j}\d f_j\in L^p(T^*X),
	\quad\text{ for every }\underline\omega=\sum_{j=1}^n\nchi_{A_j}\underline\d f_j
	\in\mathcal V.
	\]
	Using \eqref{eq:concrete_diff_wug21}, we see that
	\(\big|P(\underline\omega)\big|\leq|\underline\omega|\) holds \(\mm\)-a.e.\ for
	every \(\underline\omega\in\mathcal V\). This implies that \(P\) is well-posed and can be
	uniquely extended to a linear map
	\(P\colon\Gamma_p(T^*X)\to L^p(T^*X)\) satisfying
	the same pointwise inequality for every \(\underline\omega\in\Gamma_p(T^*X)\). In particular, the resulting map \(P\) is a morphism of	\(L^p(\mm)\)-normed \(L^\infty(\mm)\)-modules. Given that the family \(\big\{\d f\,:\,f\in N^{1,p}(X)\big\}\) generates \(L^p(T^*X)\), in order to prove that \(P\) is a submetry, it suffices to show that
	\begin{equation}\label{eq:P_quotient_aux1}
	\forall f\in N^{1,p}(X)\quad\exists\,\underline\omega\in P^{-1}(\d f):
	\quad|\underline\omega|=|Df|_p,\;\;\;\mm\text{-a.e.\ on }X.
	\end{equation}
	By Theorem \ref{thm:density_energy} there is a sequence
	\((f_n)_n\subseteq\LIP_{bs}(X)\) such that \(f_n\to f\) and
	\(\lip(f_n)\to|Df|_p\) in \(L^p(\mm)\). By \eqref{eq:concrete_diff_wug21} \((\underline\d f_n)_n\) is a bounded
	sequence. Since \(\Gamma_p(T^*X)\) is reflexive there exists
	\(\underline\omega\in\Gamma_p(T^*X)\) such that (up to a not
	relabelled subsequence)	\(\underline\d f_n\rightharpoonup\underline\omega\) in the weak topology.
	The map \(P\) is linear and continuous, thus \(\d f_n=P(\underline\d f_n)
	\rightharpoonup P(\underline\omega)\) weakly in \(L^p(T^*X)\). Since the differential \(\d\) is closed (see \cite[Theorem 2.2.9]{Gigli14}) it follows that \(P(\underline\omega)=\d f\), which implies
	\(|Df|_p=\big|P(\underline\omega)\big|\leq|\underline\omega|\) 
	\(\mm\)-a.e.\ on \(X\). Observe also that the weak convergence
	\(\underline\d f_n\rightharpoonup\underline\omega\) yields
	\[
	\int|\underline\omega|^p\,\d\mm\leq\liminf_{n\to\infty}\int|\underline\d f_n|^p
	\,\d\mm\leq\lim_{n\to\infty}\int\lip(f_n)^p\,\d\mm=\int|Df|_p^p\,\d\mm,
	\]
	so that necessarily \(|\underline\omega|=|Df|\) holds \(\mm\)-a.e.\ on \(X\).
	This proves \eqref{eq:P_quotient_aux1} and accordingly the fact that \(P\) is
	a submetry. It follows from Proposition \ref{prop:adjoint_operator}
	that \(\iota\coloneqq P^*\) is an isometric embedding of normed modules.
\end{proof}

\begin{proof}[Proof of Theorem \ref{thm:isomA}]
    Theorem \ref{thm:submetry} yields the existence of a submetry $P \colon \Gamma_{p}( T^{*}X ) \rightarrow L^{p}( T^{*}X )$ of normed modules for which $P( \underline{\d}f ) = \d f$ $\mm$-a.e.\ for every $f \in \LIP_{bs}( X )$. Theorem \ref{thm:submetry} states that the adjoint $\iota \colon L^{q}( TX ) \rightarrow \Gamma_{q}( TX )$ is an isometric embedding of normed modules.
    
    Let $f \in \LIP_{bs}(X)$ and $v \in L^{q}( T X )$. Then $\langle \d f, v \rangle (x)= \langle \underline{\d}f, \iota(v) \rangle(x)$ $\mm$-almost everywhere. Here $\langle \underline{\d}f, \iota(v) \rangle(x) = \d_{x}f( \iota(v)(x) )$ $\mm$-almost everywhere, given Definition \ref{def:concretediff}. The claim follows from this.
\end{proof}

\subsection{Isomorphism of tangent modules}\label{sec:char}

We prove Theorem \ref{thm:isomB}. Let $X=(X,d,\mm)$ be a weak LDS with an atlas $\mathscr A$, and $1<p<\infty$. In the next statement, $P\colon\Gamma_p(T^*X)\to L^p(T^*X)$ is a submetry satisfying \eqref{eq:diffext}.
\begin{proposition}\label{prop:summary}
	The following are equivalent.
	\begin{enumerate}[label=(\alph*)]
		\item\label{prop:summary.4} We have $\lip( f|_U ) = |Df|_p$ $\mm$-a.e.\ in $U$, for every chart $(U,\varphi)\in\mathscr A$ and $f\in \LIP_{bs}(X)$.
		
		\item\label{prop:summary.3} There is a collection  $\omega=\{\omega_x\}$ of moduli of continuity, such that $\lip( f|_U ) \leq \omega(|Df|_p)$ $\mm$-a.e.\ in $U$, for every chart $(U, \varphi ) \in \mathscr{A}$ and $f \in \LIP_{bs}(X)$.
		
		\item\label{prop:summary.2} The submetry $P$ is injective.
	\end{enumerate}
\end{proposition}
\begin{proof}
	Clearly (a) implies (b). Assume (b), and let $\mathcal V$ be the dense subspace in the proof of Theorem \ref{thm:submetry}. Note that (b) implies $|\underline\alpha|\le \omega(|P(\underline\alpha)|)$
	$\mm$-a.e.\ for every $\underline\alpha\in \mathcal V$ and, since $\mathcal V$ is dense and $P$ continuous, for every $\underline\alpha\in \Gamma_p(T^*X)$. Thus (b) implies (c).
	
	Assuming (c), we have that $P$ is an isometric isomorphism. Thus $|Df|_p=|P(\underline\d f)|=|\underline\d f|_{\mathscr A}^*=\lip(f|_U)$ $\mm$-a.e.\ in $U\in \mathscr A,$ cf.\ \eqref{eq:ldsnorm}.	
\end{proof}
	
\begin{remark}\label{rmk:closable}{\rm
Conditions (a)--(c) in Proposition \ref{prop:summary} are also equivalent to the closability of $\underline\d$, which implies that $\underline\d$ has a \emph{unique} extension to a bounded operator $N^{1,p}(X)\rightarrow \Gamma_{p}( T^{*}X )$. This extension and its uniqueness have been studied for PI spaces \cite{Hei:Kos:95,Fra:Haj:Kos:99,Kei:04}. Similar closability property of $\underline\d$ was applied in \cite[Section 6]{Schioppa14} to prove that, in Lipschitz differentiability spaces, the component functions of charts can be taken to be distance functions.
\fr}\end{remark}

\begin{proof}[Proof of Theorem \ref{thm:isomB}]
	Consider the isometric embedding $\iota=P^*$ obtained in Theorem \ref{thm:submetry}. Given a chart $(U,\varphi)$ of $X$, let $\mathscr A$ be an atlas containing $(U,\varphi)$. Then by Proposition \ref{prop:summary} conditions (1) and (2) in the claim are both equivalent to the submetry $P$ being injective, which in turn is equivalent to $\iota$ being an isometric isomorphism. This completes the proof.
\end{proof}

\begin{proof}[Proof of Corollary \ref{cor:reflexive}]
Since $L^q(TX)\hookrightarrow \Gamma_q(TX)$ isometrically and $\Gamma_q(TX)$ is reflexive --- recall Remark \ref{rmk:concdual} --- it follows that $L^q(TX)$ is reflexive. Consequently, $N^{1,p}(X)$ is reflexive, see \cite[Proposition 2.2.10]{Gigli14}.
\end{proof}

\begin{remark}\label{rem:generalization}{\rm
		We note that the conclusion of Corollary \ref{cor:reflexive} holds under slightly weaker assumptions. Indeed, it is sufficient to assume that the Borel set $N = X \setminus \bigcup_k U_k$ is \(p\)-negligible. With this relaxation, the inequality \eqref{eq:concrete_diff_wug21} is still valid since $|Df|_{p} = 0$ $\mm$-a.e.\ in $N$. Then the proofs of Theorem \ref{thm:isomB} and Corollary \ref{cor:reflexive} go through unchanged. Examples of spaces that are weak LDS up to a (non-trivial) $p$-negligible set appear in quasiconformal uniformization problems of metric surfaces, see \cite[Proposition 17.1]{Rajala17} for details.
		\fr}\end{remark}

For the proof of Theorem \ref{cor:BKO}, given a Borel set $Y \subset X$ we say that a bounded Borel function $\rho \colon Y \rightarrow \left[0, \infty\right)$ is a \textbf{$*$-upper gradient} of a Lipschitz function $g \colon Y \rightarrow \mathbb{R}$ if, for all compact sets $K \subset \mathbb{R}$ and Lipschitz maps $\gamma \colon K \rightarrow Y$, we have
\begin{equation}
	\label{eq:star:upper}
	| ( g \circ \gamma )' |(t) \leq \rho( \gamma_t ) | \gamma' |(t)
	\quad
	\text{ for $\mathcal{L}^{1}$-a.e.\ $t \in K$}.
\end{equation}
Here $|( g \circ \gamma )'_t|$ refers to the absolute value of the classical derivative of the Lipschitz function $g \circ \gamma \colon K \rightarrow \mathbb{R}$ and $| \gamma' |(t)$ to the metric derivative of $\gamma$,
i.e.\ $| \gamma' |(t)=\lim_{ K \ni s \rightarrow t }
	\frac{ d( \gamma_t, \gamma_s ) }{ |t-s| }$,
whenever the limit exists (and set $| \gamma' |( t ) = 0$ for isolated points $t \in K$).

\begin{proof}[Proof of Theorem \ref{cor:BKO}]
    We first observe that, for each $U_{k}$, any set $S \subset U_{k}$ porous in $U_{k}$ has null measure (recall Section \ref{subsec:porousity}). Indeed, if $S$ is as claimed, then $g(x) = d(x, S)$ satisfies $\lip( g|_{ U_{k} } )( x ) > 0$ everywhere in $S$, but the given assumptions yield that $\lip( g|_{ U_{k} } ) \leq \omega( |Dg|_{p} ) = 0$ $\mm$-a.e.\ in $S$ by the locality property \eqref{eq:sobloc}. Thus $\mm( S ) = 0$. Theorem \ref{thm:porous:infdoub} and Lemma \ref{lemm:finitedimension} imply that $U_k$ is the union of sets of finite Hausdorff dimension (up to a null-set). Thus, by further decomposing each $U_{k}$, we may assume that every $U_{k}$ is bounded, closed, and has finite Hausdorff dimension.
	
	Let $k \in \mathbb{N}$ and $g \in \LIP( U_{k} )$. We prove that 
	\begin{equation}
	    \label{eq:goal}
	    \lip( g )( x )
	    \leq
	    \omega_{x}( \rho(x) )
	    \quad
	    \text{ for $\mm$-a.e.\ $x \in U_{k}$}
	\end{equation}
for any bounded $*$-upper gradient $\rho$ of $g$.
	
Fix a Lipschitz extension $f \in \LIP_{bs}( X )$ of $g$ and a $*$-upper gradient $\rho$ of $g$. Consider $\widetilde{\rho}(x) = \nchi_{ U_{k} }(x) \rho(x) + \nchi_{ U_{k}^{c} }(x) \lip(f)(x)$. Clearly $\widetilde{\rho} \in L^{p}(\mm)$. Moreover $\widetilde{\rho}$ is an upper gradient of $f$: for any Lipschitz path $\gamma \colon \left[a, b\right] \rightarrow X$, consider the compact set $K = \gamma^{-1}( U_{k} )$. For $\mathcal{L}^{1}$-a.e.\ $t \in K$, $|( f \circ \gamma|_{K} )'|(t) = |( f \circ \gamma )'|(t)$ and $| ( \gamma|_{K} )' |(t) = | \gamma' |(t)$. Hence \eqref{eq:star:upper} yields that
	\begin{equation*}
	    | ( f \circ \gamma )' |(t)
	    \leq
	    \widetilde{\rho}( \gamma_t ) | \gamma' |(t)
	    \quad
	    \text{ for $\mathcal{L}^{1}$-a.e.\ $t \in K$}.
	\end{equation*}
    On the other hand, since $\left[a, b\right] \setminus K$ is relatively open in $\left[a, b\right]$ and $X \setminus U_{k}$ is open, we see that
    \begin{equation*}
	    | ( f \circ \gamma )' |(t)
	    \leq
	    \lip f( \gamma_t ) | \gamma' |(t)
	    \quad
	    \text{ for $\mathcal{L}^{1}$-a.e.\ $t \in \left[a, b\right] \setminus K$},
    \end{equation*}
proving that $\widetilde{\rho}$ is an upper gradient of $f$. The minimality of $| Df |_{p}$ yields that $| Df |_{p} \leq \widetilde{\rho} = \rho$ $\mm$-a.e.\ on $U_{k}$. Therefore $\lip( g ) = \lip( f|_{ U_{k} } ) \leq \omega( | Df |_{p} ) \leq \omega( \rho )$ $\mm$-a.e.\ on $U_{k}$ by the given assumptions, so \eqref{eq:goal} holds. Now \cite[Theorem 6.2, Remark 6.4]{BKO19} establishes the claim.
\end{proof}

\section{Spaces with a strongly rectifiable decomposition}\label{sec:rect}
Here we prove Theorem \ref{thm:GH} about the isometric isomorphism of the concrete and Gromov--Hausdorff tangent modules. We say that $( U, \varphi )$ is a \textbf{rectifiable chart} if $U \subset X$ is Borel, $\varphi \colon U \rightarrow \mathbb{R}^{ n }$ is a bi-Lipschitz embedding for some $n \in \mathbb{N}$, and $\varphi_{*}( \mm|_{U} )\ll\mathcal{L}^{n}|_{ \varphi(U ) }$. A metric measure space admits a \textbf{rectifiable decomposition} if there exists an atlas of rectifiable charts covering $X$ up to a null-set. We do not require an upper bound on the dimensions of the charts.

By further decomposing the domains of the bi-Lipschitz maps if needed, we may assume the rectifiable charts $(U,\varphi)$ to satisfy
\begin{equation}
    \label{eq:density}
    \frac{1}{C}\,\mathcal L^n|_{\varphi(U)}
    \leq
    \varphi_*(\mm|_U)
    \leq C\,\mathcal L^n|_{\varphi(U)},\quad\text{ for some }C\geq 1,
\end{equation}
where $n$ is the dimension of the chart, and $C$ depends on the chart.
\medskip

For the remainder of this section, we fix an isometric embedding $\iota \colon X \rightarrow \ell^{\infty}$, where $\ell^\infty$ denotes the space of bounded sequences with the standard supremum norm $\|\cdot\|$, and identify $X$ with its image. Consider a rectifiable chart $\varphi \colon U \rightarrow \mathbb{R}^{n}$ satisfying \eqref{eq:density}, and let $\psi \colon \mathbb{R}^{n} \rightarrow \ell^{\infty}$ be a Lipschitz extension of $\varphi^{-1}$. For a.e.\ $z \in \varphi(U)$ and $v \in \mathbb{R}^{n}$, we denote by
\begin{equation}
    \label{eq:metricdifferential}
    {\rm md}_{z}( \varphi^{-1} )(v)
    \coloneqq
    \lim_{ r \rightarrow 0^{+} }
        \frac{ \| \psi( z + r v ) - \psi( z ) \| }{ r }
\end{equation}
the seminorm on $\R^n$ which is the \textbf{metric derivative} of $\psi$ at $z$, see \cite[Proposition 1 and Theorem 2]{Kirchheim94}. Note that, for a.e.\ $z\in \varphi(U)$,  ${\rm md}_{z}( \varphi^{-1} )$ is independent of the chosen Lipschitz extension, and a norm (the latter fact follows from the bi-Lipschitz assumption on $\varphi^{-1} = \psi|_{U}$).

If $f \colon U \rightarrow \mathbb{R}$ is Lipschitz and $g\coloneqq f \circ \varphi^{-1} \colon \varphi( U ) \rightarrow \mathbb{R}$, by Rademacher's theorem there exists a unique linear map $D_{ z }g \colon \mathbb{R}^{n} \rightarrow \mathbb{R}$, for a.e.\ $z \in \varphi(U)$, such that
\begin{equation}
    \label{eq:concretedifferential}
    g( y ) = g( z ) + D_{ z }g (  y - z ) + o( | y - z | )
\end{equation}
for every $y \in \varphi(U)$.

One readily verifies from \eqref{eq:metricdifferential} and \eqref{eq:concretedifferential} that a space $X$ admitting a rectifiable decomposition is weak LDS. In fact, we obtain the following.
\begin{lemma}\label{lemm:rectifiable:norm}
    Suppose that $X$ admits a rectifiable decomposition and $\mathscr{A}$ is an atlas of rectifiable charts $( U, \varphi )$ that satisfy \eqref{eq:density}. If $f\in \LIP(X)$, then
	\begin{align*}
	\underline\d_xf=D_{\varphi(x)}(f\circ\varphi^{-1})\ \textrm{ and }\ |\underline\d_xf|_{\mathscr A,x}^*=\sup\big\{ D_{\varphi(x)}(f\circ\varphi^{-1})(v):\ {\rm md}_{\varphi(x)}\varphi^{-1}(v)\le 1\big\}
	\end{align*}
for $\mm$-a.e.\ $x\in U$.

In particular, $X$ is weak LDS and the pointwise norm of the concrete cotangent module is given by the dual norm of ${\rm md}( \varphi^{-1})$, for  rectifiable chart maps $\varphi$.
\end{lemma}

\begin{proof}
	For $\mm$-a.e.\ $x\in U$ the differentials $\underline\d_xf$, $D_{\varphi(x)}(f\circ\varphi^{-1})$, and the norm ${\rm md}_{\varphi(x)}\varphi^{-1}$ exist, and moreover $\varphi(x)$ is a density point of $\varphi(U)$, cf.\ \eqref{eq:density}. For such $x$ we have
	\begin{align*}
	\underline\d_xf(\varphi(y)-\varphi(x))+o(d(y,x))=f(y)-f(x)=D_{\varphi(x)}(f\circ\varphi^{-1})(\varphi(y)-\varphi(x))+o(d(y,x))
	\end{align*}
for $y\in U$. By the uniqueness of the concrete differential we obtain $\underline\d_xf=D_{\varphi(x)}(f\circ\varphi^{-1})$. Observe that $d(y,x)={\rm md}_{\varphi(x)}\varphi^{-1}(\varphi(y)-\varphi(x))+o(d(y,x))$, $y\in U$, and consequently
\begin{align*}
|\underline\d_xf|_{\mathscr A,x}^*=\lip(f|_U)(x)=&\limsup_{U\ni y\to x}\frac{\big|D_{\varphi(x)}(f\circ\varphi^{-1})(\varphi(y)-\varphi(x))\big|+o(d(y,x))}{{\rm md}_{\varphi(x)}\varphi^{-1}(\varphi(y)-\varphi(x))+o(d(y,x))}\\
\le &\sup\{|D_{\varphi(x)}(f\circ\varphi^{-1})(v):\ {\rm md}_{\varphi(x)}\varphi^{-1}(v)\le 1 \}.
\end{align*}
Since $\varphi(x)$ is a density point of $\varphi(U)$ we have that, for $\mathcal H^{n-1}$-a.e.\ $v \in \{w: {\rm md}_{\varphi(x)}\varphi^{-1}(w)=1\}$, there exists a positive sequence $ h_{i} \downarrow 0$ with $z_{i} = \varphi(x) + h_{i} v \in \varphi(U)$. Denoting $y_i\coloneqq\varphi^{-1}(z_i)\in U$ and observing that $d(y_i,x)=h_i+o(d(y_i,x))$, we obtain 
\begin{align*}
|D_{\varphi(x)}(f\circ\varphi^{-1})|(v)=\lim_{i\to \infty}\frac{|f(y_i)-f(x)|}{d(y_i,x)}\le \lip(f|_U)(x)=|\underline\d_xf|_{\mathscr A,x}^*,
\end{align*}
completing the proof.
\end{proof}

Recall that a metric measure space admits a \textbf{strongly rectifiable decomposition} if for every $\varepsilon > 0$ there exists an atlas $\mathscr{A}_{\varepsilon}$ of rectifiable charts that are $( 1 + \varepsilon )$-bi-Lipschitz.

\begin{lemma}\label{lem:Hilb_md}
	Suppose that $X$ admits a rectifiable decomposition. Then $X$ admits a strongly rectifiable decomposition if and only if for some (and thus any) atlas \(\mathscr A\) it holds that \(\mm\)-a.e.\ fiber of \(T^*_{\mathscr A}X\) is Hilbertian.
\end{lemma}
\begin{proof}
    We first prove the ``only if''-direction. For each $\varepsilon>0$, consider a rectifiable atlas $\mathscr A_\varepsilon$ of $(1+\varepsilon)$-bi-Lipschitz charts. The norm of \(\mm\)-a.e.\ fiber of $T_{\mathscr A_\varepsilon}^{*}X$
is $(1+\varepsilon)$-bi-Lipschitz to an inner product norm.
Since (the isometry classes) of the fiber norms are independent of
the chosen atlas, by sending $\varepsilon\to 0$ we obtain that the
fiber norms must in fact be induced by inner products.

Next, we claim that the ``if''-direction holds. We fix an atlas $\mathscr{A}$ and let $\delta > 0$. Consider a rectifiable chart $( U, \varphi ) \in \mathscr{A}$ of dimension $n$.

By \cite[Lemma 4]{Kirchheim94} there exists a Borel decomposition $( U_{i} )_{ i = 0 }^{ \infty }$ of $U$ and norms $( { \rm n_{i} } )_{ i = 1 }^{\infty}$ on $\mathbb{R}^{n}$ such that $\varphi|_{U_{i}} \colon U_{i} \rightarrow ( \mathbb{R}^{n}, {\rm n}_{i} )$ is a $( 1 + \delta )$-bi-Lipschitz embedding
for every $i \geq 1$ and $\mathcal{L}^{n}( \varphi( U_{0} ) ) = 0$. Hence $\mm( U_{0} ) = 0$. Up to enlarging $U_{0}$ and relabeling, we may also assume that $\mm( U_{i} ) > 0$ for every $i \geq 1$. Then for every $i \geq 1$
and $\mathcal L^n$-almost every $x \in \varphi( U_{i} )$, the metric differential of $\varphi^{-1}$ at \(x\) is $( 1 + \delta )$-comparable to the norm ${\rm n}_{i}$.
Since $\mm( U_{i} ) > 0$ and $\varphi_{*}( \mm|_{U} )\ll\mathcal{L}^{n}|_{ \varphi(U ) }$, the point $x$ can be chosen in such a way that the metric differential is induced by an inner product. After replacing ${\rm n}_{i}$ by the metric differential at such a point $x$, and relabeling the norm ${\rm n}_{i}$, the map $\varphi|_{U_{i}} \colon U_{i} \rightarrow ( \varphi( U_{i} ), {\rm n}_{i} )$ is $( 1 + \delta )^{2}$-bi-Lipschitz.
Then there exists an isometry $T_{i} \colon ( \mathbb{R}^{n}, {\rm n}_{i} ) \rightarrow \mathbb{R}^{n}$, and it follows that $\varphi_{i} = T_{i} \circ \varphi|_{ U_{i} }$ is a $( 1 + \delta )^{2}$-bi-Lipschitz chart of $X$. 
Since $\mm( U_{0} ) = 0$ and $( U, \varphi ) \in \mathscr{A}$ was an arbitrary chart, we can construct this way
a $( 1 + \delta )^{2}$-bi-Lipschitz atlas for $X$ from $\mathscr{A}$. The claim follows by arbitrariness of \(\delta>0\).
\end{proof}
We recall from the introduction that given a rectifiable decomposition $\left\{ A_{k} \right\}_{ k = 1 }^{\infty}$ up to a negligible set, with each $A_{k}$ $k$-rectifiable, we let $T_{GH} X = \bigsqcup_{ k = 1 }^{ \infty } A_{k} \times \mathbb{R}^{ k }$. We endow each $A_{k} \times \mathbb{R}^{ k }$ with the Euclidean norm. Following \cite[Definition 4.5]{GP16}, we introduce the following notion. 
\begin{definition}[Gromov--Hausdorff tangent module]
	Let \((X,d,\mm)\) admit a strongly rectifiable decomposition. We define the \textbf{Gromov--Hausdorff tangent module}
	\(L^2(T_{GH}X)\) as the space of all those measurable sections \(v\)
	of \(T_{GH}X\) (considered up to \(\mm\)-a.e.\ equality) that satisfy
	\[
	\int\big|v(x)\big|^2\,\d\mm(x)<+\infty.
	\]
	The space \(L^2(T_{GH}X)\) is a \(L^2(\mm)\)-normed \(L^\infty(\mm)\)-module
	if endowed with the pointwise norm
	\[
	|v|(x)\coloneqq\big|v(x)\big|,\quad\text{ for }\mm\text{-a.e.\ }x\in X.
	\]
\end{definition}
As a consequence of the previously discussed results, we can provide an
alternative proof of Theorem \ref{thm:GH}, which is one of the main
results of \cite{GP16} (namely \ \cite[Theorem 5.1]{GP16}).
\begin{proof}[Proof of Theorem \ref{thm:GH}]
	Let \(\mathscr A=\big\{(U_k,\varphi_k)\big\}_{k\in\N}\) be an atlas
	on \((X,d,\mm)\). Without loss of
	generality we may assume that each set \(U_k\) is compact. Given any \(k\in\N\), let us fix
	an orthonormal basis \(v^k_1,\ldots,v^k_{n_k}\) for \(\Gamma_2(T_{\mathscr A}X)\)
	on \(U_k\), namely, for every \(j,\ell=1,\ldots,n_k\) it holds that
	\(\langle v^k_j,v^k_\ell\rangle=\delta_{j\ell}\) in the \(\mm\)-a.e.\ sense
	on \(U_k\). The elements \(v^k_j\) can be seen as \(\mm\)-a.e.\ defined
	maps \(v^k_j\colon U_k\to\R^{n_k}\). For \(\mm\)-a.e.\ \(x\in U_k\), we
	denote by \(\phi_x\colon\R^{n_k}\to\R^{n_k}\) the unique linear isomorphism
	with \(\phi_x\big(v^k_j(x)\big)={\rm e}_j\) for all \(j=1,\ldots,n_k\),
	where \(\{{\rm e}_1,\ldots,{\rm e}_{n_k}\}\) stands for the canonical basis of
	\(\R^{n_k}\). Let us define the operator
	\({\rm I}\colon\Gamma_2(T_{\mathscr A}X)\to L^2(T_{GH}X)\) as follows:
	given any \(v\in\Gamma_2(T_{\mathscr A}X)\), we set
	\[
	{\rm I}(v)(x)\coloneqq\phi_x\big(v(x)\big)\in\R^{n_k},
	\quad\text{ for every }k\in\N\text{ and }\mm\text{-a.e.\ }x\in U_k.
	\]
	To check that \({\rm I}(v)\) is (the equivalence class of) a measurable
	section of \(T_{GH}X\), observe that
	\[
	\phi_x\big(v(x)\big)=\sum_{k\in\N}\nchi_{U_k}(x)\,\phi_x\bigg(\sum_{j=1}^{n_k}
	\big\langle v(x),v^k_j(x)\big\rangle\,v^k_j(x)\bigg)
	=\sum_{k\in\N}\nchi_{U_k}(x)\sum_{j=1}^{n_k}\big\langle v(x),
	v^k_j(x)\big\rangle\,{\rm e}_j
	\]
	holds for \(\mm\)-a.e.\ \(x\in X\). Moreover, \(\mm\)-a.e.\ it holds that
	\[
	\big|{\rm I}(v)\big|^2=\sum_{k\in\N}\nchi_{U_k}\bigg|\sum_{j=1}^{n_k}
	\langle v,v^k_j\rangle\,v^k_j\bigg|^2=\sum_{k\in\N}\nchi_{U_k}
	\sum_{j=1}^{n_k}\langle v,v^k_j\rangle^2
	=\sum_{k\in\N}\nchi_{U_k}\bigg|\sum_{j=1}^{n_k}\langle v,v^k_j\rangle\,
	{\rm e}_j\bigg|^2=|v|^2.
	\]
	This grants that \(\rm I\) maps \(\Gamma_2(T_{\mathscr A}X)\) to
	\(L^2(T_{GH}X)\) and preserves the pointwise norm. Being
	\(\phi_x\) a linear isomorphism for \(\mm\)-a.e.\ \(x\in X\),
	we deduce that \(\rm I\)  is an isometric isomorphism of \(L^2(\mm)\)-normed
	\(L^\infty(\mm)\)-modules. Therefore, calling \(\iota\colon L^2(TX)\to
	\Gamma_2(T_{\mathscr A}X)\) the isometric embedding given by Theorem
	\ref{thm:isomA}, we conclude that the composition
	\(\mathscr I\coloneqq{\rm I}\circ\iota\) is an isometric embedding.
\end{proof}

\section{Lipschitz differentiability spaces}\label{sec:lds}
In this section we prove Theorem \ref{thm:module:LDS}. To this end, we note that in any metric measure space $X$ whose porous sets have zero measure, for every $f \in \LIP(X)$ and $U \subset X$ Borel, $\lip( f|_{U} ) = \lip( f )$ $\mm$-a.e.\ on $U$ \cite[Proposition 2.8]{BS13}. In particular, Lipschitz extensions of weak charts self-improve to strong charts (compare \eqref{eq:differential:OG} and \eqref{eq:lds}). Since porous sets in Lipschitz differentiability spaces are negligible by \cite[Theorem 2.4]{BS13}, it follows that $X$ is LDS if and only if it is weak LDS and porous sets have null measure.
\begin{proof}[Proof of Theorem \ref{thm:module:LDS}]
By Theorem \ref{cor:BKO}, $X$ is weak LDS satisfying \(\emph{(1)}\) in Theorem \ref{thm:isomB} for every $p' \geq p$. The proof is complete after we verify that porous sets are negligible. To this end, let $S\subset X$ be porous in $X$. By Remark \ref{rmk:por} we have $\lip( g )(x)>0$ for every $x\in S$, where $g(y)\coloneqq d(y,S)$, while $\lip(g) \leq \omega( |Dg|_{p} )$ yields $\lip(g)= 0$ $\mm$-a.e.\ on $S$ by \eqref{eq:sobloc}. Hence $\mm(S) = 0$.
\end{proof}

We discuss some known conditions implying \eqref{eq:BKO}. Firstly, as noted in the introduction, being a Lipschitz differentiability space does not imply \eqref{eq:BKO}. On the other hand, having a doubling measure and $p$-Poincar\'e inequality does, by the results in \cite{Cheeger00}.

In fact, the doubling condition and Poincar\'e inequality can be substantially weakened. A suitable asymptotic version of the Poincar\'e inequality (with comparable notions introduced in \cite{BateLi18} and also \cite{EB19}) imply \eqref{eq:BKO}. We say that $X$ has an {\bf asymptotic non-homogeneous $p$-Poincar\'e inequality} (asymptotic $p$-NPI), if the measure $\mm$ vanishes on porous sets, and $X$ has a countable partition $\{B_i\}$ up to a null set, with constants $\lambda_i>0$, moduli of continuity $\omega_{i}$, and $\epsilon_i\colon[0,\infty)\to \R$ satisfying $\lim_{t\to 0}\epsilon_i(t)/t=0$, so that
\begin{equation}\label{eq:asymptotic:poincare}
\aint{ B(x,r) }
| f - f_{ B(x,r) } |
\,\d\mm
\leq
r\,\omega_{i}\left(
\left(    
\aint{ B( x, \lambda_{i}r ) }
|Df|_p^{p}
\,\d\mm
\right)^{1/p}\right)
+
\epsilon_{i}( r )
\end{equation}
for $\mm$-a.e.\ $x \in B_i$ and $f \in \LIP_{bs}(X)$. Here \(f_{B(x,r)}\) stands for the average integral \(\aint{B(x,r)}f\,\d\mm\). 

\begin{proposition}\label{prop:sufcon}
Let $1<p<\infty$. Suppose $X$ has an asymptotic $p$-Poincar\'e inequality. Then $X$ satisfies \eqref{eq:BKO}.
\end{proposition}
\begin{proof}
Assume $X$ has an asymptotic $p$-NPI. Using  \cite[Lemma 4.10]{BateLi18} and Lebesgue's differentiation theorem -- both of which are applicable since porous sets have zero $\mm$-measure -- we conclude from \eqref{eq:asymptotic:poincare} that $\lip f(x) \le C_{x} \omega_i( |Df|_p )(x)$ $\mm$-a.e.\ $x \in B_i$ for any $f\in\LIP_{bs}(X)$, for some $C_{x}$ independent of $f$, establishing \eqref{eq:BKO}.
\end{proof}
\begin{remark}{\rm
As we see from the proof, we can allow $\omega_{i}$, $\lambda_{i}$ and $\epsilon_{i}$ in \eqref{eq:asymptotic:poincare} to depend on the point $x \in X$ even within each $B_{i}$, as long as this is independent of $f \in \LIP_{bs}(X)$.
\fr}\end{remark}
\noindent\textbf{Acknowledgements.}
The first named author was supported by the Academy of Finland,
project number 308659, and by the Vilho, Yrj\"{o} and Kalle V\"{a}is\"{a}l\"{a} Foundation.
The second named author was supported by the Academy of Finland, project
number 314789, and by the Balzan project led by Prof.\ Luigi Ambrosio.
The third named author was supported by the Swiss National Foundation, grant no.\ 182423. The authors wish to thank Sylvester Eriksson-Bique for helpful discussions and the reviewers for their useful suggestions.
\def\cprime{$'$} \def\cprime{$'$}


\begin{thebibliography}{10}

\bibitem{AmbrosioGigliSavare11-3}
{\sc L.~Ambrosio, N.~Gigli, and G.~Savar{\'e}}, {\em Density of {L}ipschitz
  functions and equivalence of weak gradients in metric measure spaces}, Rev.
  Mat. Iberoam., 29 (2013), pp.~969--996.

\bibitem{Bate15}
{\sc D.~Bate}, {\em Structure of measures in {L}ipschitz differentiability
  spaces}, J. Amer. Math. Soc., 28 (2015), pp.~421--482.

\bibitem{BKO19}
{\sc D.~Bate, I.~Kangasniemi, and T.~Orponen}, {\em Cheeger's differentiation
  theorem via the multilinear {K}akeya inequality},  (2019).
\newblock arXiv:1904.00808.

\bibitem{BateLi17}
{\sc D.~Bate and S.~Li}, {\em Characterizations of rectifiable metric measure
  spaces}, Annales scientifiques de l'\'{E}cole normale sup\'{e}rieure, 50
  (2017), pp.~1--37.

\bibitem{BateLi18}
\leavevmode\vrule height 2pt depth -1.6pt width 23pt, {\em Differentiability
  and {P}oincar\'{e}-type inequalities in metric measure spaces}, Adv. Math.,
  333 (2018), pp.~868--930.

\bibitem{BS13}
{\sc D.~Bate and G.~Speight}, {\em Differentiability, porosity and doubling in
  metric measure spaces}, Proceedings of the American Mathematical Society, 141
  (2013), pp.~971--985.

\bibitem{Bjorn-Bjorn11}
{\sc A.~Bj{\"o}rn and J.~Bj{\"o}rn}, {\em Nonlinear potential theory on metric
  spaces}, vol.~17 of EMS Tracts in Mathematics, European Mathematical Society
  (EMS), Z\"urich, 2011.

\bibitem{Cheeger00}
{\sc J.~Cheeger}, {\em Differentiability of {L}ipschitz functions on metric
  measure spaces}, Geom. Funct. Anal., 9 (1999), pp.~428--517.

\bibitem{DiMarinoSpeight15}
{\sc S.~Di~Marino and G.~Speight}, {\em The $p$-weak gradient depends on $p$},
  Proceedings of the American Mathematical Society, 143 (2015), pp.~5239--5252.

\bibitem{EB19}
{\sc S.~Eriksson-Bique}, {\em Characterizing spaces satisfying {P}oincar\'{e}
  {I}nequalities and applications to differentiability}, Geometric and
  Functional Analysis, 29 (2019), pp.~119--189.

\bibitem{Fra:Haj:Kos:99}
{\sc B.~Franchi, P.~Haj{\l}asz, and P.~Koskela}, {\em Definitions of {S}obolev
  classes on metric spaces}, Ann. Inst. Fourier (Grenoble), 49 (1999),
  pp.~1903--1924.

\bibitem{Gigli14}
{\sc N.~Gigli}, {\em Nonsmooth differential geometry - {A}n approach tailored
  for spaces with {R}icci curvature bounded from below}, Memoirs of the
  American Mathematical Society, 251 (2017).

\bibitem{Gigli17}
\leavevmode\vrule height 2pt depth -1.6pt width 23pt, {\em {L}ecture notes on
  differential calculus on $\sf {R}{C}{D}$ spaces},  (2018).
\newblock Publ. RIMS Kyoto Univ. 54.

\bibitem{GP16}
{\sc N.~Gigli and E.~Pasqualetto}, {\em Equivalence of two different notions of
  tangent bundle on rectifiable metric measure spaces}.
\newblock Accepted at Communications in Analysis and Geometry,
  arXiv:1611.09645.

\bibitem{GP19}
{\sc N.~Gigli and E.~Pasqualetto}, {\em Differential structure associated to
  axiomatic {S}obolev spaces}, Expositiones Mathematicae,  (2019).

\bibitem{Hei:Kos:95}
{\sc J.~Heinonen and P.~Koskela}, {\em Weighted {S}obolev and {P}oincar\'{e}
  inequalities and quasiregular mappings of polynomial type}, Math. Scand., 77
  (1995), pp.~251--271.

\bibitem{HKST15}
{\sc J.~Heinonen, P.~Koskela, N.~Shanmugalingam, and J.~Tyson}, {\em Sobolev
  spaces on metric measure spaces: An approach based on upper gradients},
  Cambridge University Press, United States, 1 2015.

\bibitem{Keith04}
{\sc S.~Keith}, {\em A differentiable structure for metric measure spaces},
  Advances in Mathematics, 183 (2004), pp.~271 --315.

\bibitem{Kei:04}
{\sc S.~Keith}, {\em Measurable differentiable structures and the
  {P}oincar\'{e} inequality}, Indiana Univ. Math. J., 53 (2004),
  pp.~1127--1150.

\bibitem{Kell19}
{\sc M.~Kell}, {\em On {C}heeger and {S}obolev differentials in metric measure
  spaces}, Rev. Mat. Iberoam., 35 (2019), pp.~2119--2150.

\bibitem{Kirchheim94}
{\sc B.~Kirchheim}, {\em Rectifiable {M}etric {S}paces: {L}ocal {S}tructure and
  {R}egularity of the {H}ausdorff {M}easure}, Proceedings of the American
  Mathematical Society, 121 (1994), pp.~113--123.

\bibitem{MMPZ03}
{\sc M.~E. Mera, M.~Mor\'{a}n, D.~Preiss, and L.~Zaj\'{i}\v{c}ek}, {\em
  Porosity, $\sigma$-porosity and measures}, Nonlinearity, 16 (2003),
  pp.~247--255.

\bibitem{pas-luc18}
{\sc E.~Pasqualetto and D.~Lu\v{c}i\'{c}}, {\em The {S}erre-{S}wan theorem for normed
  modules}, Rendiconti del Circolo Matematico di Palermo, Series 2, 68 (2019),
  p.~385–404.

\bibitem{Rajala17}
{\sc K.~Rajala}, {\em Uniformization of two-dimensional metric surfaces},
  Inventiones mathematicae, 207 (2017), pp.~1301--1375.

\bibitem{Schioppa14}
{\sc A.~Schioppa}, {\em On the relationship between derivations and measurable
  differentiable structures}, Ann. Acad. Sci. Fenn. Math., 39 (2014),
  pp.~275--304.

\bibitem{Schioppa16:A}
\leavevmode\vrule height 2pt depth -1.6pt width 23pt, {\em Derivations and
  {A}lberti representations}, Adv. Math., 293 (2016), pp.~436--528.

\bibitem{Schioppa16:B}
\leavevmode\vrule height 2pt depth -1.6pt width 23pt, {\em Metric currents and
  {A}lberti representations}, J. Funct. Anal., 271 (2016), pp.~3007--3081.

\end{thebibliography}
\end{document}